\documentclass[a4paper]{article}
\usepackage[utf8]{inputenc}
\usepackage[UKenglish]{babel}
\usepackage{amsmath,amssymb,amsthm}
\usepackage{bbm}
\usepackage{mathpartir}
\usepackage{stmaryrd}

\newtheorem{thm}{Theorem}

\newtheorem{lem}[thm]{Lemma}

\theoremstyle{definition}
\newtheorem{defn}[thm]{Definition}

\newtheorem{exa}[thm]{Example}

\newcommand{\declarecategory}[2]{\expandafter\newcommand\csname#1\endcsname{\ensuremath{\mathsf{#2}}}}
\declarecategory{fin}{Fin}
\declarecategory{set}{Set}
\declarecategory{balg}{BAlg}
\declarecategory{finba}{BAlg_{fin}}
\declarecategory{ealg}{EAlg}
\declarecategory{bker}{BKer}
\declarecategory{chaus}{CHaus}

\newcommand{\CatC}{\mathcal{C}}
\newcommand{\FinKer}{\mathsf{FinKer}}
\newcommand{\ProDet}{\mathsf{ProDet}}
\newcommand{\Det}{\mathsf{det}}
\newcommand{\Pro}{\mathsf{Pro}}
\newcommand{\id}{\mathsf{id}}
\newcommand{\samp}{\mathsf{samp}}
\newcommand{\Kleisli}{\mathsf{Kleisli}}
\newcommand{\Radon}{R}

\newcommand{\op}{{}^\mathsf{op}}
\newcommand{\interval}{\ensuremath{\mathbb I}}
\newcommand{\two}{\mathbbm 2}
\newcommand{\kerto}{\rightsquigarrow}

\DeclareMathOperator{\riesz}{\mathsf{M}}

\newcommand{\copymor}{\mathsf{copy}}

\begin{document}
\title{A causal Markov category with Kolmogorov products}
\author{Sean Moss and Sam Staton}

\date{Notes from October 2021, compiled \today}

\maketitle

In
\cite{fritz-rischel-infinite-products-and-zero-one-laws-in-categorical-probability}
(Problem 6.7) the problem was posed of finding an interesting Markov category which is \emph{causal} and has all (small) \emph{Kolmogorov products}.
Here we give an example where the deterministic subcategory is the category of Stone spaces (i.e.\ the dual of the category of Boolean algebras) and the kernels correspond to a restricted class of Kleisli arrows for the Radon monad.

We look at this from two perspectives. First via pro-completions and Stone spaces directly. Second via duality with Boolean and algebras and effect algebras.

\section{Perspective in terms of pro-completions}
Let $\CatC$ be a Markov category, with deterministic wide subcategory
$\Det(\CatC)$. A motivating example is $\CatC=\FinKer$, so that
$\Det(\CatC)=\fin$, and $\Pro(\Det(\CatC))$ is the category of Stone
spaces.

We often write `$X\to X'$' for an object of $\Pro(\CatC)$, as a shorthand for $X\cong \lim_{X\to X'} X'$, where the $X'$s are in $\CatC$. 

Let $\ProDet(\CatC)$ be the full subcategory of the profinite
completion $\Pro(\CatC)$ given by limits of cofiltered diagrams that factorize through $\Det(\CatC)$.  
Equivalently, $\ProDet(\CatC)$ is the full image of
\[
  \Pro(\Det(\CatC)) \longrightarrow \Pro(\CatC),
\]
where $\Pro(\Det(\CatC))$ is cartesian monoidal.  
$\ProDet(\CatC)$ is monoidal with
\[
  X \otimes Y = \lim_{\substack{X \to X' \\ Y \to Y' \\ \text{in } \Pro(\Det(\CatC))}} X' \otimes Y',
\]
i.e.\ the cartesian product $X \times Y$ in $\Pro(\Det(\CatC))$.  
Every $X \in \ProDet(\CatC)$ has a copy structure given by
\[
  X \longrightarrow X' \xrightarrow{\mathrm{copy}} X' \otimes X'
\]
for each $X \to X'$ in $\Pro(\Det(\CatC))$.

\begin{lem}
$\ProDet(\CatC)$ is a Markov category with deterministic subcategory $\Pro(\Det(\CatC))$.
\end{lem}

\begin{proof}
The monoidal unit is the terminal object $1 \in \CatC$.  
Clearly every map in $\Pro(\Det(\CatC))$ is deterministic.  
Let $f : X \to Y$ be deterministic in $\ProDet(\CatC)$.  
Then for every $Y \to Y'$ in $\Pro(\Det(\CatC))$, there is an $X \to X'$ in $\Pro(\Det(\CatC))$ such that
\[
  X \xrightarrow{f} Y \to Y' \;=\; X \to X' \xrightarrow{f'} Y'
\]
for some $f' : X' \to Y'$ in $\CatC$, whence we have equations
\begin{align*}
  X &\xrightarrow{f} Y \xrightarrow{\mathrm{copy}} Y \otimes Y \to Y' \otimes Y' \\
    &= X \xrightarrow{f} Y \to Y' \xrightarrow{\mathrm{copy}} Y' \otimes Y' \\
    &= X \to X' \xrightarrow{f'} Y' \xrightarrow{\mathrm{copy}} Y' \otimes Y',
\end{align*}
and
\begin{align*}
  X &\xrightarrow{\mathrm{copy}} X \otimes X \xrightarrow{f \otimes f} Y \otimes Y \to Y' \otimes Y' \\
    &= X \xrightarrow{\mathrm{copy}} X \otimes X \to X' \otimes X' \xrightarrow{f' \otimes f'} Y' \otimes Y' \\
    &= X \to X' \xrightarrow{\mathrm{copy}} X' \otimes X' \xrightarrow{f' \otimes f'} Y' \otimes Y'.
\end{align*}
Thus for some factorization $X \to X'' \to X'$ in $\Pro(\Det(\CatC))$ we have
\[
  X'' \to X' \xrightarrow{f'} Y' \xrightarrow{\mathrm{copy}} Y' \otimes Y' 
  = X'' \to X' \xrightarrow{\mathrm{copy}} X' \otimes X' \xrightarrow{f' \otimes f'} Y' \otimes Y',
\]
whence we deduce that
\[
  X'' \to X' \xrightarrow{f'} Y'
\]
is in $\Det(\CatC)$.  
Since $Y \to Y'$ was arbitrary, the map $f$ is in $\Pro(\Det(\CatC))$.
\end{proof}

\begin{thm}
$\ProDet(\CatC)$ has all small Kolmogorov products.
\end{thm}

\begin{proof}
Given $(X_i : i \in I)$, let $X$ be their product in $\Pro(\Det(\CatC))$, as an object of $\ProDet(\CatC)$.  
Then $X$ is an infinite tensor product of the $X_i$, since $\Pro(\Det(\CatC)) \to \ProDet(\CatC)$ preserves cofiltered limits, essentially by construction.
\end{proof}

In the case of $\ProDet(\FinKer)$ there is also a faithful functor $\ProDet(\FinKer)\to \Kleisli(\Radon)$ into the Kleisli category of the Radon monad on compact Hausdorff spaces. On objects, this maps a cofiltered diagram to its limit in compact Hausdorff spaces: a Stone space.
Recall that a morphism $f\colon \lim_i X_i\to \lim_j Y_j$ in $\ProDet(\FinKer)$
is given by, for each $j$, a choice $i$ and a morphism in $\FinKer(X_i,Y_j)$ making suitable diagrams commute, and modulo choice of~$i$. 
Thus for each deterministic point $x:1\to \lim_i X_i$, i.e.~point of the Stone space, 
we have a morphism $p_j(x)\in \FinKer(1,Y_j)$ for all $j$, i.e.~a distribution $p_j(x)$ on $Y_j$ for all $j$, and these are all suitably compatible;
we then construct a function $p:\lim_iX_i\to R(\lim_j Y_j)$ between compact Hausdorff spaces, mapping $x$ to the Kolmogorov extension of $p_j(x)$ on the Stone space $\lim_{j}X_j$. 

From \cite[Example 11.35]{fritz-a-synthetic-approach-to-markov-kernels}, the Markov category of all probability kernels between measurable spaces is causal. And so, from this faithful embedding, we see that $\ProDet(\FinKer)$ is causal.

We revisit this in the Section~\ref{sec:dual}, where the algebra brings out a different angle, and also suggests a direct proof of causality. 

\subsection{Aside: Additional basic results}

\begin{lem}\label{lem:rep}
If $\CatC$ is representable, then so is $\ProDet(\CatC)$.
\end{lem}

\begin{proof}
Suppose $\Det(\CatC) \to \CatC$ has a right adjoint $D : \CatC \to \Det(\CatC)$.  
By $2$-functoriality of $\Pro$, this gives a right adjoint $D^+$ to $\Pro(\Det(\CatC)) \to \Pro(\CatC)$:
\[
  D^+ X = \lim_{X \to X' \in \Pro(\Det(\CatC))} D X'.
\]
\end{proof}

\begin{lem}\label{lem:asrep}
If $\CatC$ is a.s.-compatibly representable, then so is $\ProDet(\CatC)$.
\end{lem}

\begin{proof}
Let $p : T \to A$ and $f,g : A \to X$ be maps in $\ProDet(\CatC)$, and suppose that $f =_{\!p\text{-a.s.}} g$, i.e.
\[
  T \xrightarrow{p} A \to A \otimes A \xrightarrow{\id \otimes f} A \otimes X
  = T \xrightarrow{p} A \to A \otimes A \xrightarrow{\id \otimes g} A \otimes X,
\]
or equivalently,
\begin{align*}
  &T \xrightarrow{p} A \to A \otimes A \xrightarrow{\id \otimes f^\sharp} A \otimes D^+X \xrightarrow{\id \otimes \samp} A \otimes X
\\  =\quad& T \xrightarrow{p} A \to A \otimes A \xrightarrow{\id \otimes g^\sharp} A \otimes D^+X \xrightarrow{\id \otimes \samp} A \otimes X.
\end{align*}
We are required to show that we can cancel the factor of $\id \otimes \samp$ from each side.  
It is straightforward to show that this is possible upon post-composing with projections $A \otimes X \to A' \otimes X'$.
\end{proof}
(Note: Lemmas~\ref{lem:rep} and~\ref{lem:asrep} are not relevant to $\FinKer$.)

\begin{lem}
Let $\CatC = \FinKer$.  
Then for $X,Y \in \FinKer$, $L \in \ProDet(\FinKer)$, every kernel $p : X \kerto Y \otimes L$ admits a conditional $k : X \otimes Y \kerto L$.
\end{lem}

\begin{proof}
Without loss of generality, let $X = 1$.  
For $y \in Y$, the value of $p(\ast)(y,\top^Z)$ is independent of the deterministic projection $\pi : L \to Z \in \Pro(\fin)$ chosen.  
If it is zero, we define $k(\ast,y)$ arbitrarily.  
If it is not zero, say $0 < K \le 1$, then for each $\pi : L \to Z$ we have a $K$-valued measure $p(\ast)(y,-)$ on $Z$ forming a cone of kernels over the diagram for $L$.  
Normalizing by $K$, this is still a cone and hence we have defined a kernel $k : 1 \otimes Y \kerto L$.

We must check that this really is the conditional, but this is straightforward.
\end{proof}

\section{Dual perspective in terms of Boolean algebras and effect algebras}
\label{sec:dual}
\paragraph{Boolean algebras and effect algebras}

Let \balg\ denote the category of Boolean algebras, whose subcategory \finba\ of finite algebras is equivalent to $\fin\op$, the dual of the category of finite sets and all functions.
We recall some generalities.
The embedding
\begin{displaymath}
  \balg \hookrightarrow [\finba\op,\set] \simeq [\fin,\set]
\end{displaymath}
is full and faithful, and the essential image is the closure of the representable functors under filtered colimits.\

The category \ealg\ of \emph{effect algebras} is traditionally presented in terms of partial algebras, but it is sometimes helpful to use the following viewpoint: the mapping
\begin{displaymath}
  (E,\ovee,0,1) \mapsto \lambda n. \{ (e_1,\ldots,e_n) \in E^n : e_1 \ovee \ldots \ovee e_n = 1 \}
\end{displaymath}
embeds \ealg\ as a full reflective subcategory of $[\fin,\set]$ containing \balg
\begin{displaymath}
  \balg \hookrightarrow \ealg \hookrightarrow [\fin,\set].
\end{displaymath}
We won't refer to the precise limit-preservation conditions that characterize \ealg\ this way \cite{jacobs-convexity-duality-effects,staton-uijlen-effect-algebras-presheaves-nonlocality-contextuality}.
The cocartesian monoidal structure on finite Boolean algebras extends by Day convolution to a symmetric monoidal structure $\boxplus$ on $[\fin,\set]$, which restricts to the coproduct on \balg.
It's not obvious, but the reflection of the Day convolution into \ealg\ gives the classical effect algebra tensor $\otimes$ when restricted to \ealg.
While the inclusion $\balg \hookrightarrow \ealg$ does not preserve coproducts in general, it does preserve the initial Boolean algebra $\two = \{\bot,\top\}$.
This makes \ealg\ an interesting semicocartesian category.

\paragraph{The interval effect monoid}

Let \interval\ denote the effect algebra whose underlying set is the real interval $[0,1]$ with partial sum $a \oplus b$ defined whenever $a + b \leq 1$ in which case it is equal to $a + b$.
As a functor $\fin \to \set$,
\begin{displaymath}
  \interval(n) = \{ \phi \in [0,1]^n : \sum_{i = 1}^n \phi(i) = 1 \}
\end{displaymath}
and the action of a function $f : m \to n$ sends $\phi \in \interval(m)$ to
\begin{displaymath}
  \lambda i \in n. \sum_{f(j) = i} \phi(j).
\end{displaymath}

As is well-known, \interval\ is a monoid with respect to $\otimes$, where the multiplication can be represented in terms of the Day convolution
\begin{displaymath}
  \int^{a,b \in \fin} \fin(a \times b,n) \times \interval(a) \times \interval(b) \cong (\interval\boxplus\interval)(n) \to \interval(n)
\end{displaymath}
by the dinatural transformation
\begin{align*}
  \fin(a \times b,n) \times \interval(a) \times \interval(b) & \to \interval(n) \\
  (f,\phi,\psi) & \mapsto \interval(f)(\lambda (i,j) \in a \times b.  \phi(i)\psi(j))
\end{align*}
and the unit is just the unique map $\two \to \interval$.
(Recall that $\interval \boxplus \interval$ is not itself the effect tensor $\interval \otimes \interval$, rather it's an object of $[\fin,\set]$ whose reflection into \ealg\ is $\interval \otimes \interval$).

\paragraph{Measures on and kernels between Boolean algebras}

We will only work with probability measures.
A measure (i.e.\ valuation) on a Boolean algebra $A$ is a map $m : A \to \interval$ in \ealg, i.e. a natural transformation.
This is actually just the usual notion of finitely-additive measure (or modular valuation).

\begin{defn}
  For $A,B \in \balg$, a \emph{kernel} $k : A \kerto B$ is a map $k : B \to \interval \otimes A$.
\end{defn}

Modules for the effect monoid \interval\ have been called ``convex effect algebras'' and have a representation theorem in terms of ordered vectors spaces with a chosen strong unit \cite{gudder-pulmannova-representation-theorem-for-convex-effect-algebras}.
For our purposes, the following special case is helpful.

\begin{lem}
  Let $A \in \balg$.
  Then $\interval \otimes A$ is isomorphic to the effect algebra $\riesz(A)$ whose elements are locally constant functions $f : S(A) \to [0,1]$ from the Stone space $S(A)$ of $A$ to the unit interval with the evident partial sum operation.
\end{lem}

Stone duality tells us that $\balg\op$ is equivalent to a full subcategory of \chaus, i.e.\
\begin{displaymath}
  \balg(B,A) \cong \chaus(S(A),S(B)).
\end{displaymath}
It is fairly straightforward to check that measures on $A \in \balg$ are equivalent to Radon (probability) measures on $S(A)$.
Thus we `extend' Stone duality a little bit with the observation
\begin{displaymath}
  \ealg(B,\interval) \cong \chaus(1,R(S(B)))
\end{displaymath}
where $R$ is the Radon monad.
In general, we only get an injection
\begin{displaymath}
  \ealg(B,\interval \otimes A) \hookrightarrow \chaus(S(A),R(S(B))).
\end{displaymath}
The Radon kernels $k : S(A) \to R(S(B))$ that arise in this way are those with the property that, for every clopen set $C \subseteq S(B)$, the function $\lambda x \in S(A). k(x,C)$ is locally constant --- in that case there is a (necessarily finite) partition of $S(A)$ into clopen sets such that the measure assigned to $C$ is constant on each.

\begin{exa}
  A Radon kernel between Stone spaces that does not arise in this way is $b : 2^{\mathbb N} \kerto 2$ where $b(\vec x)$ is the Bernoulli distribution with bias $\sum_{i \in \mathbb N} 2^{-i-1}x_i$.
\end{exa}

\paragraph{A Markov category of Boolean kernels}

\begin{defn}
  Let $\bker = ((\balg)_{\interval \otimes (-)})\op$, the dual of the Kleisli category of the monad $\interval \otimes (-)$ on \ealg\ restricted to the objects of \balg.
  Thus the objects of \bker\ are Boolean algebras and the morphism $A \to B$ are kernels, i.e. maps $B \to \interval \otimes A$ in \ealg.
\end{defn}

\begin{lem}
  \bker\ is a semicartesian monoidal category.
\end{lem}
\begin{proof}
  The tensor of effect algebras straightforwardly induces a tensor of free \interval-modules.
  \bker\ is semicartesian because the monoidal unit is $\two$, the initial object of \ealg.
\end{proof}

The copy-discard structure of \bker\ is given in the evident way.
For a Boolean algebra $A$, the `codiscard' map is the unique map
\begin{displaymath}
  \two \to \interval \otimes A
\end{displaymath}
and the `cocopy' map is
\begin{displaymath}
  A \otimes A \to \interval \otimes A
\end{displaymath}
given the composing the codiagonal in \balg\ with the unit of \interval\ in \ealg.
\begin{displaymath}
  A \otimes A \cong A +_\balg A \to A \cong \two \otimes A \to \interval \otimes A
\end{displaymath}

\begin{lem}
  The deterministic subcategory of \bker\ is precisely $\balg\op \hookrightarrow \bker$.
\end{lem}
\begin{proof}
  Let $f : B \to \interval \otimes A$ represent a deterministic map $A \to B$ of \bker.
  We want to show that, for all $b \in B$, $f(b)$ has the form $1 \otimes a$ for some $a \in A$.
  Write $f(b)$ in the form $\alpha_1 \otimes a_1 + \ldots + \alpha_k \otimes a_k$ with the $a_i$ mutually orthogonal and $0 \leq \alpha_i \leq 1$.
  Copyability implies
  \begin{displaymath}
    \alpha_1 \otimes a_1 + \ldots + \alpha_k \otimes a_k = \alpha_1^2 \otimes a_1 + \ldots + \alpha_k^2 \otimes a_k
  \end{displaymath}
  whence $\alpha_i \in \{0,1\}$, as required.
\end{proof}

\paragraph{Causality}

From \cite[Example 11.35]{fritz-a-synthetic-approach-to-markov-kernels}, the Markov category of all probability kernels between measurable spaces is causal.
Since \bker\ embeds faithfully into it with a morphism of Markov categories, \bker\ is also causal.
Let's also argue directly.

Suppose we are given
\begin{mathpar}
  f : B \to \interval \otimes A
  \and
  g : C \to \interval \otimes B
  \and
  h_i : D \to \interval \otimes C
\end{mathpar}
in \ealg\ for $i = 1,2$ with $A,B,C,D \in \balg$ such that the two composites
\begin{displaymath}
  (h_i \otimes 1_C) \circ \copymor_C \circ g \circ f : A \to D \otimes C
\end{displaymath}
in \bker\ agree.
We are required to show that the two composites
\begin{displaymath}
  (((h_i \otimes 1_C) \circ \copymor_C \circ g) \otimes 1_B) \circ \copymor_B \circ f : A \to D \otimes C \otimes B
\end{displaymath}
are equal.
The latter are actually maps $D \otimes C \otimes B \to \interval \otimes A$, hence it suffices to check equality on generators of the form $d \otimes c \otimes b$.
By decomposing $c$, we can assume that $c$ is `small' enough that both $h_i(d) \in \interval \otimes C$ are sums of the form $\sum_i \alpha_i x_i$ with, for each $i$, either $c \leq x_i$ or $c \wedge x_i = 0_C$.
In terms of $\riesz(C)$, this says that both $h_i(d)$ are constant on the clopen subset of $S(C)$ corresponding to $c$.
It follows that the maps $(h_i \otimes 1_C) \circ \copymor_C : C \to D \otimes C$ send, as functions $D \otimes C \to \interval \times C$, the element $d \otimes c$ to $\xi_i \otimes c$ for some $\xi_i \in [0,1]$.
By hypothesis, either $g \circ f : A \to C$ sends, as a function $C \to \interval \otimes A$, the element $c$ to $0$ or $\xi_1 = \xi_2$.
In either case it is clear to see that the required equation holds.

\paragraph{Kolmogorov products}

Recall that \balg\ is the closure under filtered colimits of the representable functors in $[\fin,\set]$, and that by construction the Day tensor preserves colimits in each argument.
It follows that the inclusion $\balg \hookrightarrow \ealg$ preserves filtered colimits, as does the (strict) monoidal functor $\balg \to \bker\op$, and hence that \bker\ has all infinite tensor products of all small families, given by the coproduct in \balg\ of the underlying objects.
Obviously, the marginals are deterministic, i.e.\ \balg\ has all Kolmogorov products.

\paragraph{See also }Lorenzin and Zanasi, \emph{Approaching the Continuous from the Discrete: an Infinite Tensor Product Construction}, 2025 \cite{lorenzin-zanasi}.

\bibliographystyle{plain}
\bibliography{../../papers,../../books}

\end{document}